\documentclass{amsart}
\usepackage[utf8]{inputenc}

\usepackage{amsthm,amssymb}
\usepackage{tikz-cd}
\usepackage{hyperref}
\usepackage{cleveref}

\usepackage{todonotes}

\newtheorem{theorem}{Theorem}
\newtheorem{lemma}[theorem]{Lemma}
\theoremstyle{definition}
\newtheorem{definition}[theorem]{Definition}
\theoremstyle{remark}

\newcommand{\co}{\colon\thinspace}
\renewcommand{\phi}{\varphi}

\DeclareMathOperator{\Spec}{Spec}

\newcommand{\CatOf}[1]{\mathsf{#1}}

\title{There aren't that many Morava $E$--theories}
\author{Kiran Luecke}\author{Eric Peterson}

\begin{document}

\maketitle

\begin{abstract}
    Let $k$ be a perfect field of characteristic $p$. Associated to any (1-dimensional, commutative) formal group law of finite height $n$ over $k$ there is a complex oriented cohomology theory represented by a spectrum denoted $E(n)$ and commonly referred to as Morava $E$-theory. These spectra are known to admit $E_\infty$-structures, and the dependence of the $E_\infty$-structure on the choice of formal group law has been well studied (cf.\ [GH], [R], [L], Section 5, [PV]). In this note we show that the underlying homotopy type of $E(n)$ is independent of the choice of formal group law.
\end{abstract}

\section{Introduction}
In this section we collect some standard results in chromatic homotopy theory that will be used in the proof. A combined reference for (almost) all of them is [P]. All rings in this note are commutative, graded, and concentrated in even degree. Rings that are commonly ungraded (e.g. a perfect characteristic $p$ field $k$ and its $p$-typical Witt vectors $W(k)$) are viewed graded rings concentrated in degree zero.
\begin{definition}
In light of the above, define a \textit{formal group law} over a ring $R$ to be a (graded) ring map $MU_*\rightarrow R$. Recall that $MU_*$ is the coefficient ring of the complex bordism spectrum $MU$, and $MU_*\simeq\mathbb{Z}[a_1,a_2,...]$, $|a_i|=2i$. 
Define an \textit{ungraded formal group law} over $R$ to be a (graded) ring map $\mathrm{Laz}\rightarrow R$, $\mathrm{Laz}\simeq\mathbb{Z}[b_1,b_2,...]$, $|b_i|=0$.
\end{definition}

\begin{theorem}\label{LEFT} ([Lan], cf.\ also [H])
Let $f\co MU_*\rightarrow R$ be a formal group law over a ring $R$. Let $\mathcal{M}_{MU}$ be the stack associated to the the Hopf algebroid $(MU_*,MU_*MU)$. It admits a canonical map $\Spec MU_*\rightarrow \mathcal{M}_{MU}$. Then the assignment
$$X\mapsto MU_*X\otimes_{MU_*}R$$
defines a homology theory if and only if the composite
$$\Spec R\xrightarrow{f}\Spec MU_*\rightarrow \mathcal{M}_{MU}$$
is flat.
\end{theorem}

\begin{theorem}\label{brownrep} ([A])
Every homology theory is represented by a spectrum, which is unique up to homotopy equivalence. Every morphism of homology theories is represented by some morphism of representing spectra. This morphism of spectra is well-defined up to phantom maps.
\end{theorem}

\begin{theorem}\label{nophantom} ([HS])
Let $MU_*\rightarrow E_*$ and $MU_*\rightarrow F_*$ be two Landweber exact formal group laws and let $E$ and $F$ be spectra representing the homology theories $E_*X= MU_*X\otimes_{MU_*}E_*$ and $F_*X= MU_*X\otimes_{MU_*}F_*$. There are no phantom maps $E\rightarrow F$.
\end{theorem}

\begin{theorem}\label{lubintate} ([LT])
Let $k$ be a perfect field of characteristic $p$. Let $f\co \mathrm{Laz}\rightarrow k$ be an ungraded formal group law of height $n$. Define the ring $$LT(k)=W(k)[[u_1,u_2,\ldots,u_{n-1}]][\beta^\pm], \quad |\beta|=2.$$ There is a Landweber exact formal group law $$LT(f)\co MU_*\rightarrow LT(k)$$ satisfying the following:
\begin{enumerate}
\item $LT(f)$ is a universal deformation of $f$ up to $\star$--isomorphism.
\item If $g\co \mathrm{Laz}\rightarrow k$ is an ungraded formal group law that is isomorphic to $f$, then $LT(g)$ can be chosen to be isomorphic to $LT(f)$.
\item If $t\co k\rightarrow k'$ is a map of perfect characteristic $p$ fields and $T\co LT(k)\rightarrow LT(k')$ denotes the map which is $W(t)\co W(k)\rightarrow W(k')$ on coefficients and sends $u_i$ to $u_i$ and $\beta$ to $\beta$, then $LT(t \circ f)$ can be chosen such that $T\circ LT(f)=LT(t\circ f)$.
\end{enumerate}
\end{theorem}

\begin{theorem}\label{lazard} ([Laz] Theorem IV)
All ungraded formal group laws of a fixed height $n$ over an algebraically closed field of positive characteristic are isomorphic.
\end{theorem}

\section{The Proof}
\begin{lemma}
Let $R$ be a ring with two Landweber exact formal group laws $e,f\co MU_*\rightarrow R$, and let $E$ and $F$ be the spectra corresponding to $e$ and $f$ under the Landweber exact functor theorem. If there is a ring extension $u\co  R\rightarrow S$ which is split as an $R$-module map and over which the formal group laws $u\circ e$ and $u\circ f$ are isomorphic and Landweber exact, then $E$ and $F$ have the same homotopy type.
\end{lemma}

\begin{proof}
First we show that $E\simeq F$ if and only if the map $\eta\co E_*\rightarrow F_*E$ (induced by $1\wedge id\co  \mathbb{S}\wedge E\rightarrow F\wedge E$) is split as a map of $F_*$-modules. Note that $E_*$ is canonically an $F_*$ module by the \textit{equality} $E_*=R=F_*$. We claim that $F_*E$ is flat over $F_*$. Indeed, interpret Landweber exactness as flatness of maps to $\mathcal{M}_{MU}$ (cf.\ Theorem \ref{LEFT}). Consider the following pullback diagram
$$\begin{tikzcd}
\Spec (F_*E)\arrow[d]\arrow[r] &\Spec(E_*)\arrow[d,"e"] \\
 \Spec (F_*)\arrow[r,"f"] &\mathcal{M}_{MU} .
\end{tikzcd}$$ The right vertical map is flat by the assumption that $e$ is Landweber exact, hence so is the left vertical map as flat maps are preserved under pullback, i.e., $F_*E$ is flat over $F_*$. It follows that $(F\wedge E)_*X=F_*X\otimes_{F_*}F_*E$. For an $F_*$-module map $\xi\co F_*E\rightarrow F_*$ consider, for each spectrum $X$, the diagram
$$\begin{tikzcd}
 E_*X\arrow[rr,"(1\wedge id)\circ(-)"]& &(F\wedge E)_*X=F_*X\otimes_{F_*}F_*E\arrow[d,"id\otimes \xi"]\\
  & & F_*X\otimes_{F_*}F_* .
\end{tikzcd}$$ 
The diagonal composite $t_\xi\co E_*X\rightarrow F_*X$ defines a morphism of homology theories. Applying Brown representability for maps of homology theories (cf.\ Theorem \ref{brownrep}) and the absence of phantom maps between Landweber exact homotopy types (cf.\ Theorem \ref{nophantom}) produces a map of spectra $T_\xi\co E\rightarrow F$ well-defined up to homotopy. The assignment $\xi\mapsto T_\xi$ defines a map $\phi\co  \CatOf{Modules}_{F_*}(F_*E,F_*)\rightarrow F^*E$. Now suppose a splitting of $\eta$ exists and call it $s$. Then there is a map $\phi(s)\co E\rightarrow F$ which on homotopy groups (set $X=\mathrm{pt}$ above) is the composite of $\eta$ and $s$, which is the identity map, so $E\simeq F$. In the other direction, if $E\simeq F$ then $F_*E\simeq E_*E$, so it suffices to split $(1\wedge id)_*\co E_*\rightarrow E_*E$, which is split by the multiplication map $E\wedge E\rightarrow E$. 

Now we show that $\eta\co E_*\rightarrow F_*E$ is split if there is a ring extension $u\co  R\rightarrow S$ which is split as an $R$-module map, and such that the formal group laws $u\circ e$ and $u\circ f$ are isomorphic and Landweber exact. Write $g:=u\circ e$ and $h:=u\circ f$. Assume they are isomorphic and Landweber exact. Let $G$ and $H$ be the spectra associated to $g$ and $h$. Then we have the following commutative diagram of $F_*(=R)$-modules
$$\begin{tikzcd}
S=G_* \arrow[rr,"1\otimes 1\otimes id"] & & H_*G= S\otimes_{MU_*}MU_*MU\otimes_{MU_*}S\\
R=E_*\arrow[u,"u"]\arrow[rr,"1\otimes 1\otimes id"]  & & F_*E=R\otimes_{MU_*}MU_*MU\otimes_{MU_*}R\arrow[u,"u\otimes id\otimes u"]  .
\end{tikzcd}$$

Since the right vertical map is split (as a map of $F_*$-modules), splitting the bottom horizontal map (as a map of $F_*$-modules) is equivalent to splitting the diagonal map (as a map of $F_*$-modules). Moreover, the top horizontal map is split (as an $S=H_*$-algebra (and hence $F_*$-algebra) map in fact) by first choosing an isomorphism of $g$ and $h$, which induces an $H_*$-algebra isomorphism $H_*G\simeq G_*G$ and then using the multiplication map $G\wedge G\rightarrow G$ to split $G_*\rightarrow G_*G$. Therefore splitting the diagonal map is equivalent to splitting the left vertical map (as a map of $F_*$-modules).
\end{proof}

\begin{lemma} At height $n<\infty$ the homotopy type of Morava $E$-theory depends only on the choice of a perfect characteristic $p$ field. In other words, let $k$ be a perfect characteristic $p$ field and let $E(n)_1$ and $E(n)_2$ be two Morava $E$-theory spectra corresponding to two arbitrary height $n< \infty$ ungraded formal group laws $e_1$ and $e_2$ over $k$. Then as spectra, $E(n)_1\simeq E(n)_2$.
\end{lemma}
\begin{proof}
It suffices to check the conditions of the lemma. Let $\overline{k}$ be the algebraic closure of $k$. Let $R$ and $S$ be the (2-periodic) Lubin-Tate deformation rings $$R:=W(k)[[u_1,...,u_{n-1}]][\beta^\pm]$$ $$S:=W(\overline{k})[[u_1,...,u_{n-1}]][\beta^\pm].$$ Let $e$ and $f$ be formal group laws over $R$ universally deforming $e_1$ and $e_2$. Then $e$ and $f$ are Landweber exact (cf.\ Theorem \ref{lubintate}) and the corresponding spectra are the Morava $E$-theory spectra $E(n)_1$ and $E(n)_2$ (cf.\ [P] Definition 3.5.4). Let $u\co R\rightarrow S$ be the map that extends coefficients by $W(k)\rightarrow W(\overline{k})$, sends $u_i$ to $u_i$, and $\beta$ to $\beta$. Then the formal group laws $g:=u\circ e$ and $h=u\circ f$ over $S$ are isomorphic and Landweber exact (again, cf.\ Theorem \ref{lubintate}), since they universally deform the isomorphic formal group laws $u\circ e$ and $u\circ f$ over $\overline{k}$ (cf.\ Theorem \ref{lazard}).  Finally, the map $u\co R\rightarrow S$ is split as an $R$-module map, because the map $W(k)\rightarrow W(\overline{k})$ is split as a $W(k)$-module map (cf.\ [S]).
\end{proof}

\section{Acknowledgements}
We would like to thank Rodrigo Marlasca Aparicio for posting the Mathoverflow question ``Does the spectrum of Morava E-theory depend only on height?" [Ap] which prompted this note. We would also like to thank Will Sawin for providing us with the statement we needed regarding the splitting of $p$-typical Witt vectors (in private communication as well as the Mathoverflow post [S]).

\section{References}

\noindent
[A] Adams, J.F.
A variant of E. H. Brown's representability theorem,
Topology,
Volume 10, Issue 3,
1971,
Pages 185-198,
ISSN 0040-9383,

\ 

\noindent
[Ap] Aparicio, R. Mathoverflow question. https://mathoverflow.net/questions/412686/does-the-spectrum-of-morava-e-theory-depend-only-on-height?noredirect=1{}\&{}lq=1

\

\noindent
[GH] Goerss, P. and M. Hopkins. Moduli Spaces of Commutative Ring Spectra

\

\noindent
[H] Hohnhold, H. The Landweber exact functor theorem, Math. Surveys Monogr., vol. 201, Amer. Math. Soc.,
Providence, RI, 2014, pp. 35–45. MR 3328537

\

\noindent
[HS] Hovey, M. and Strickland, N. P. Morava K-theories and localisation, Mem.
Amer. Math. Soc., 139 (1999), no. 666.

\

\noindent
[Lan] Landweber, P. Homological Properties of Comodules over MU${}_*$(MU) and BP${}_*$(BP). American Journal of Mathematics, Volume 98, No. 3, pp. 591-610, 1976

\

\noindent
[Laz] Lazard, M. Sur les groupes de Lie formels a un parametre. \textit{Bull. Soc. Math. France}, 83:251-274, 1955.

\

\noindent
[LT] Lubin, J. and Tate, J. Formal moduli for one-parameter formal LIe groups. \text{Bull. Soc. Math. France}, 94:49-59, 1966.

\

\noindent
[L] Lurie, J. Elliptic Cohomology II: Orientations.

\

\noindent
[P] Peterson, E. Formal geometry and bordism operations. \textit{Cambridge studies in advanced mathematics}. 177, 2019.

\

\noindent
[PV]  Pstragowski P. and Vankoughnett, P. Abstract Goerss-Hopkins Theory

\

\noindent
[R] Rezk, C. Notes on the Hopkins-Miller Theorem

\

\noindent
[S] Sawin, W. Mathoverflow answer. https://mathoverflow.net/questions/413773/splitting-the-witt-vectors-of-overline-mathbbf-p

\end{document}